\documentclass[12pt,amsymb,fullpage]{amsart}
\usepackage{amssymb,amscd,pstricks}

\newtheorem{theorem}{Theorem}[section]
\newtheorem{defn}[theorem]{Definition}

\newtheorem{lemma}[theorem]{Lemma}

\newtheorem{eple}[theorem]{Example}
\newtheorem{rmk}[theorem]{Remarks}
\newtheorem{dsc}[theorem]{Discussion}
\newtheorem{nota}[theorem]{Notation}

\newsavebox{\indbin}
\savebox{\indbin}{\begin{picture}(0,0)
\newlength{\gnu}
\settowidth{\gnu}{$\smile$} \setlength{\unitlength}{.5\gnu}
\put(-1,-.65){$\smile$} \put(-.25,.1){$|$}
\end{picture}}

\newcommand{\be}{\begin{enumerate}}
\newcommand{\bd}{\begin{defn}}
\newcommand{\bt}{\begin{theorem}}
\newcommand{\bl}{\begin{lemma}}
\newcommand{\ee}{\end{enumerate}}
\newcommand{\ed}{\end{defn}}
\newcommand{\et}{\end{theorem}}
\newcommand{\el}{\end{lemma}}

\begin{document}
\title{Solving the Heat Equation using Nonstandard Analysis}
\author{Tristram de Piro}
\address{Mathematics Department, Harrison Building, Streatham Campus, University of Exeter, North Park Road, Exeter, Devon, EX4 4QF, United Kingdom}
 \email{tdpd201@exeter.ac.uk}
\maketitle
\begin{abstract}
We use the nonstandard Fourier transform method, see \cite{dep1}, along with an established nonstandard approach to ODE's, see \cite{cut} and \cite{dep2}, to find a solution to the heat equation, on $(0,\infty)\times\mathcal{R}$, with a given boundary condition $g$ at $t=0$. We use this result to find an algorithm, converging to a solution of this equation, with applications to derivatives pricing in finance.
\end{abstract}

We adopt the following notation;\\

\begin{defn}
\label{heatdefs}
For $\eta\in{{^{*}{\mathcal N}}\setminus{\mathcal N}}$, we let $(\overline{\mathcal{R}_{\eta}},\mathfrak{C}_{\eta},\lambda_{\eta})$ be as in Definition 0.15 of \cite{dep1}.\\

We let $(\overline{\mathcal{R}_{\eta}},L(\mathfrak{C}_{\eta}),L(\lambda_{\eta}))$ denote the associated Loeb space, see Definition 0.5 of \cite{dep1}.\\

$(\mathcal{R},\mathfrak{B},\mu), (\mathcal{R}^{+-\infty},\mathfrak{B}',\mu')$ are as in Lemma 0.6 of \cite{dep1}.\\

$\overline{\mathcal{T}_{\eta}}=\{\tau\in{^{*}{\mathcal R}_{\geq 0}}:0\leq\tau<\eta\}$ and we again denote by $\mathfrak{C}_{\eta}$, the restriction of $\mathfrak{C}_{\eta}$ to $\overline{\mathcal{T}_{\eta}}$, and $\lambda_{\eta}$ the restriction of the counting measure.\\

$(\overline{\mathcal{T}_{\eta}},L(\mathfrak{C}_{\eta}),L(\lambda_{\eta}))$ is the corresponding Loeb space.\\

$\mathcal{T}=\mathcal{R}_{\geq 0}$ and $(\mathcal{T},\mathfrak{B},\mu), (\mathcal{T}^{+\infty},\mathfrak{B}',\mu')$ are defined analogously to Lemma 0.6 of \cite{dep1}.\\

$(\overline{\mathcal{T}_{\eta}}\times \overline{\mathcal{R}_{\eta}},\mathfrak{C}_{\eta}^{2},\lambda_{\eta}^{2})$ is as in Definition 0.15 of \cite{dep1}.\\

$(\overline{\mathcal{T}_{\eta}}\times \overline{\mathcal{R}_{\eta}},L(\mathfrak{C}_{\eta}^{2}),L(\lambda_{\eta}^{2}))$ is the corresponding Loeb space.\\

$(\overline{\mathcal{T}_{\eta}}\times \overline{\mathcal{R}_{\eta}},L(\mathfrak{C}_{\eta})\times L(\mathfrak{C}_{\eta}),L(\lambda_{\eta})\times L(\lambda_{\eta}))$ is the complete product of the Loeb spaces $(\overline{\mathcal{T}_{\eta}},L(\mathfrak{C}_{\eta}),L(\lambda_{\eta}))$ and $(\overline{\mathcal{R}_{\eta}},L(\mathfrak{C}_{\eta}),L(\lambda_{\eta}))$.\\

Similarly, $(\mathcal{T}^{+\infty}\times \mathcal{R}^{+-\infty},\mathfrak{B}'\times\mathfrak{B}',\mu'\times\mu')$ and
$(\mathcal{T}\times \mathcal{R},\mathfrak{B}\times\mathfrak{B},\mu\times\mu)$ are the complete products of $(\mathcal{T}^{+\infty},\mathfrak{B}',\mu')$, $(\mathcal{R}^{+-\infty},\mathfrak{B}',\mu')$ and $(\mathcal{T},\mathfrak{B},\mu)$, $(\mathcal{R},\mathfrak{B},\mu)$ respectively.\\

We let $({^{*}\mathcal{R}},{^{*}\mathfrak{D}})$ denote the hyperreals, with the transfer of the Borel field $\mathfrak{D}$ on $\mathcal{R}$. A function $f:(\overline{\mathcal{R}_{\eta}},\mathfrak{C}_{\eta})\rightarrow ({^{*}\mathcal{R}},{^{*}\mathfrak{D}})$ is measurable, if $f^{-1}:{^{*}\mathfrak{D}}\rightarrow \mathfrak{C}_{\eta}$. Similarly, $f:(\overline{\mathcal{T}_{\eta}}\times\overline{\mathcal{R}_{\eta}},\mathfrak{C}_{\eta}^{2})\rightarrow ({^{*}\mathcal{R}},{^{*}\mathfrak{D}})$ is measurable, if $f^{-1}:{^{*}\mathfrak{D}}\rightarrow \mathfrak{C}_{\eta}^{2}$. Observe that this is equivalent to the definition given in \cite{Loeb}. We will abbreviate this notation to $f:\overline{\mathcal{R}_{\eta}}\rightarrow{^{*}\mathcal{R}}$ or $f:\overline{\mathcal{T}_{\eta}}\times\overline{\mathcal{R}_{\eta}}\rightarrow{^{*}\mathcal{R}}$ is measurable, $(*)$. The same
applies to $({^{*}\mathcal{C}},{^{*}\mathfrak{D}})$, the hyper complex numbers, with the transfer of the Borel field $\mathfrak{D}$, generated by the complex topology. Observe that $f:\overline{\mathcal{R}_{\eta}}\rightarrow{^{*}\mathcal{C}}$ or $f:\overline{\mathcal{T}_{\eta}}\times\overline{\mathcal{R}_{\eta}}\rightarrow{^{*}\mathcal{C}}$ is measurable, in this sense, iff $Re(f)$ and $Im(f)$ are measurable in the sense of $(*)$.

\end{defn}

We have the following lemma, generalising Theorem 0.7 of \cite{dep1} and Theorem 22 of \cite{and};\\

\begin{lemma}
\label{squaremmp}

The identity;\\

$i:(\overline{\mathcal{T}_{\eta}}\times \overline{\mathcal{R}_{\eta}}, L(\mathfrak{C}_{\eta}^{2}),L(\lambda_{\eta}^{2}))\rightarrow (\overline{\mathcal{T}_{\eta}}\times \overline{\mathcal{R}_{\eta}}, L(\mathfrak{C}_{\eta})\times L(\mathfrak{C}_{\eta}), L(\lambda_{\eta})\times L(\lambda_{\eta}))$\\

and the standard part mapping;\\

$st:(\overline{\mathcal{T}_{\eta}}\times \overline{\mathcal{R}_{\eta}}, L(\mathfrak{C}_{\eta})\times L(\mathfrak{C}_{\eta}), L(\lambda_{\eta})\times L(\lambda_{\eta}))$\\

$\indent \ \ \ \ \ \ \ \ \ \ \ \ \ \ \ \ \ \ \ \ \ \ \ \ \ \ \ \ \rightarrow(\mathcal{T}^{+\infty}\times \mathcal{R}^{+-\infty},\mathfrak{B}'\times\mathfrak{B}',\mu'\times\mu')$\\

are measurable and measure preserving.

 \end{lemma}

\begin{proof}
To show that $i$ is measurable and measure preserving, it is sufficient to prove that;\\

(i). $L(\mathfrak{C}_{\eta})\times L(\mathfrak{C}_{\eta})\subset L(\mathfrak{C}_{\eta}^{2})$.\\

(ii). $L(\lambda_{\eta}^{2})|_{L(\mathfrak{C}_{\eta})\times L(\mathfrak{C}_{\eta})}=
L(\lambda_{\eta})\times L(\lambda_{\eta})$.\\

As in \cite{and}, if $A\in\mathfrak{C}_{\eta}$;\\

$\{M\in\sigma(\mathfrak{C}_{\eta}):M\times A\in\sigma(\mathfrak{C}_{\eta}\times\mathfrak{C}_{\eta})\}$\\

is a $\sigma$-algebra, containing $\mathfrak{C}_{\eta}$, hence, it equals $\sigma(\mathfrak{C}_{\eta})$. Similarly, if $B\in\sigma(\mathfrak{C}_{\eta})$;\\

$\{M\in\sigma(\mathfrak{C}_{\eta}):B\times M\in\sigma(\mathfrak{C}_{\eta}\times\mathfrak{C}_{\eta})\}$\\

is a $\sigma$-algebra, and equals $\sigma(\mathfrak{C}_{\eta})$. Therefore;\\

$\mathfrak{C}_{\eta}\times\mathfrak{C}_{\eta}\subset \sigma(\mathfrak{C}_{\eta})\times\sigma(\mathfrak{C}_{\eta})=\sigma(\mathfrak{C}_{\eta}\times\mathfrak{C}_{\eta})$\\

Now, using Ward Henson's result, see footnote 1 of \cite{dep1}, it follows that $L(\lambda_{\eta}^{2})=L(\lambda_{\eta})\times L(\lambda_{\eta})$ on $\sigma(\mathfrak{C}_{\eta})\times\sigma(\mathfrak{C}_{\eta})$, $(*)$. Now, suppose that $\{C,D\}\subset L(\mathfrak{C}_{\eta})$ then, there exists $\{C_{1},C_{2},D_{1},D_{2}\}\subset\sigma(\mathfrak{C}_{\eta})$, with $C_{1}\subset C\subset C_{2}$, $D_{1}\subset D\subset D_{2}$, $L(\lambda_{\eta})({C_{2}\setminus C_{1}})=0$, $L(\lambda_{\eta})({D_{2}\setminus D_{1}})=0$, $(**)$, and $C_{1}\times D_{1}\subset C\times D\subset C_{2}\times D_{2}$. Moreover, $({C_{2}\times D_{2}\setminus C_{1}\times D_{1}})\subset (({C_{2}\setminus C_{1}})\times D_{2})\cup (C_{2}\times ({D_{2}\setminus D_{1}}))$, $(***)$. By $(*),(**),(***)$, $L(\lambda_{\eta}^{2})({C_{2}\times D_{2}\setminus C_{1}\times D_{1}})=0$. Therefore, $C\times D\in L(\mathfrak{C}_{\eta}^{2})$, and the product $\sigma$-algebra $L(\mathfrak{C}_{\eta})\times L(\mathfrak{C}_{\eta})\subset L(\mathfrak{C}_{\eta}^{2})$, $(\dag)$. Using $(*),(\dag)$, $L(\lambda_{\eta}^{2})$ agrees with $L(\lambda_{\eta})\times L(\lambda_{\eta})$ on this algebra, hence, the complete product $L(\mathfrak{C}_{\eta})\times L(\mathfrak{C}_{\eta})\subset L(\mathfrak{C}_{\eta}^{2})$, showing $(i)$, and $L(\lambda_{\eta}^{2})|_{L(\mathfrak{C}_{\eta})\times L(\mathfrak{C}_{\eta})}=
L(\lambda_{\eta})\times L(\lambda_{\eta})$, by the definition of a completion, showing $(ii)$.\\

We recall the result, Theorem 0.7, of \cite{dep1}, that;\\

$st:(\overline{\mathcal{R}_{\eta}}, L(\mathfrak{C}_{\eta}), L(\lambda_{\eta}))\rightarrow (\mathcal{R}^{+-\infty},\mathfrak{B}',\mu')$\\

is measurable and measure preserving, $(\sharp)$. Similarly, one can show that;\\

$st:(\overline{\mathcal{T}_{\eta}}, L(\mathfrak{C}_{\eta}), L(\lambda_{\eta}))\rightarrow (\mathcal{T}^{+\infty},\mathfrak{B}',\mu')$\\

is measurable and measure preserving, $(\sharp\sharp)$. The rest of the argument is fairly straightforward, if $\{B_{1},B_{2}\}\subset\mathfrak{B}'$, then, using $(\sharp),(\sharp\sharp)$, $st^{-1}(B_{1}\times B_{2})\in L(\mathfrak{C}_{\eta})\times L(\mathfrak{C}_{\eta})$, and $L(\mathfrak{C}_{\eta})\times L(\mathfrak{C}_{\eta})(st^{-1}(B_{1}\times B_{2}))=\mu'\times\mu'(B_{1}\times B_{2})$. It follows, using the usual argument, as in the first part of the proof, that the push forward measure $st_{*}(L(\mathfrak{C}_{\eta})\times L(\mathfrak{C}_{\eta}))$ agrees with $\mu'\times\mu'$ on $\mathfrak{B}'\times\mathfrak{B}'$, considered as a product $\sigma$-algebra. Then, the result follows easily from the definition of a complete product.

\end{proof}

The following definition is based on Definition 0.18 of \cite{dep1};\\

\begin{defn}{Discrete Partial Derivatives}\\
\label{partials}

Let $f:\overline{\mathcal{T}_{\eta}}\times\overline{\mathcal{R}_{\eta}}\rightarrow{^{*}\mathcal{C}}$ be measurable. Then we define ${\partial f\over\partial t}$ to be the unique measurable function satisfying;\\

${\partial f\over\partial t}({j\over\eta},x)=\eta(f({j+1\over\eta},x)-f({j\over\eta},x))$ for $j\in{^{*}\mathcal{N}}_{0\leq j\leq{\eta^{2}-2}}, x\in\overline{\mathcal{R}_{\eta}}$\\

${\partial f\over\partial t}({\eta^{2}-1\over\eta},x)=0$\\

${\partial f\over\partial x}(t,{j\over\eta})=\eta(f(t,{j+1\over\eta})-f(t,{j\over\eta}))$ for $j\in{^{*}\mathcal{N}}_{-\eta^{2}\leq j\leq{\eta^{2}-2}}, t\in\overline{\mathcal{T}_{\eta}}$\\

${\partial f\over\partial x}(t,{\eta^{2}-1\over\eta})=0$\\
\end{defn}

\begin{rmk}
\label{rmkpartials}
If $f$ is measurable, then so are ${\partial f\over\partial t}$, ${\partial f\over\partial x}$ and ${\partial^{2} f\over\partial x^{2}}$. This follows immediately, by transfer, from the corresponding result for the discrete derivatives of discrete functions $f:{\mathcal{T}_{n}}\times{\mathcal{R}_{n}}\rightarrow\mathcal{C}$, where $n\in{\mathcal{N}}$, see Definition 0.15 and Definition 0.18 of \cite{dep1}.
\end{rmk}

\begin{lemma}
\label{nsheat}
Given a measurable boundary condition $g:\overline{\mathcal{R}_{\eta}}\rightarrow{^{*}\mathcal{C}}$, there exists a unique measurable $f:\overline{\mathcal{T}_{\eta}}\times\overline{\mathcal{R}_{\eta}}\rightarrow{^{*}\mathcal{C}}$, satisfying the nonstandard heat equation;\\

${\partial f\over\partial t}-{\partial^{2}f\over\partial x^{2}}=0$ on $({\overline{\mathcal{T}_{\eta}}\setminus [{\eta^{2}-1\over\eta},\eta)})\times\overline{\mathcal{R}_{\eta}}$\\

with $f(0,x)=g(x)$, for $x\in\overline{\mathcal{R}_{\eta}}$, $(*)$.

\end{lemma}

\begin{proof}
Observe that, by Definition \ref{partials}, if $f:\overline{\mathcal{T}_{\eta}}\times\overline{\mathcal{R}_{\eta}}\rightarrow{^{*}\mathcal{C}}$ is measurable, then;\\

${\partial^{2}f\over\partial x^{2}}(t,{j\over\eta})={\eta^{2}}(f(t,{j+2\over\eta})-2f(t,{j+1\over\eta})+f(t,{j\over\eta}))$, $(-\eta^{2}\leq j\leq{\eta^{2}-3})$.\\

${\partial^{2}f\over\partial x^{2}}(t,{\eta^{2}-2\over\eta})=-{\eta^{2}}(f(t,{\eta^{2}-1\over\eta})-f(t,{\eta^{2}-2\over\eta}))$\\

${\partial^{2}f\over\partial x^{2}}(t,{\eta^{2}-1\over\eta})=0$\\

Therefore, if $f$ satisfies $(*)$, we must have;\\

$f({i+1\over\eta},{j\over\eta})=f({i\over\eta},{j\over\eta})+\eta(f({i\over\eta},{j+2\over\eta})-2f({i\over\eta},{j+1\over\eta})+f({i\over\eta},{j\over\eta}))$,\\

\indent \ \ \ \ \ \ \ \ \ \ \ \ \ \ \ \ \ \ \ \ \ \ \ \ \ \  $(0\leq i\leq{\eta^{2}-2}, -\eta^{2}\leq j\leq{\eta^{2}-3})$.\\

$f({i+1\over\eta},{\eta^{2}-2\over\eta})=f({i\over\eta},{\eta^{2}-2\over\eta})-\eta(f({i\over\eta},{\eta^{2}-1\over\eta})-f({i\over\eta},{\eta^{2}-2\over\eta}))$,\\

\indent \ \ \ \ \ \ \ \ \ \ \ \ \ \ \ \ \ \ \ \ \ \ \ \ \ \ \ $(0\leq i\leq{\eta^{2}-2})$.\\

$f({i+1\over\eta},{\eta^{2}-1\over\eta})=f({i\over\eta},{\eta^{2}-1\over\eta})$, $(0\leq i\leq{\eta^{2}-2})$.\\

$f(0,{j\over\eta})=g({j\over\eta})$, $(-\eta^{2}\leq j\leq{\eta^{2}-1})$. $(**)$\\

If $\eta=n\in\mathcal{N}$, then given any measurable $g:{\mathcal{R}_{n}}\rightarrow\mathcal{C}$, the condition $(**)$, clearly determines a unique measurable, see Definition 0.15 of \cite{dep1}, $f:{\mathcal{T}_{n}}\times{\mathcal{R}_{n}}\rightarrow\mathcal{C}$, satisfying $(*)$. As the condition $(*)$ can be written down uniformly, in Robinson's higher order logic, we obtain the result, immediately, by transfer.

\end{proof}

\begin{defn}
\label{nstransf}
We recall the definition from \cite{dep1}, Definition 0.15. Given a measurable $f:\overline{\mathcal{T}_{\eta}}\times\overline{\mathcal{R}_{\eta}}\rightarrow{^{*}\mathcal{C}}$, we define $exp_{\eta}(-\pi ixy)$ and $exp(\pi ixy)$ to be the $\mathfrak{C}_{\eta}^{2}$ measurable counterparts of the transfers of $exp(\pi ixy)$ and $exp(-\pi ixy)$ to $\overline{\mathcal{R}_{\eta}}^{2}$. We define the nonstandard Fourier transform in space;\\

$\hat{f}(t,y)=\int_{\overline{\mathcal R}_{\eta}}f(t,x)exp_{\eta}(-\pi ixy)d\lambda_{\eta}(x)$\\

and the nonstandard inverse Fourier transform in space;\\

$\check{f}(t,y)=\int_{\overline{\mathcal R}_{\eta}}f(t,x)exp_{\eta}(\pi ixy)d\lambda_{\eta}(x)$\\

As in Definition 0.20 of \cite{dep1}, we let $\phi_{\eta},\psi_{\eta}:\overline{\mathcal{R}_{\eta}}\rightarrow{^{*}\mathcal{C}}$ be defined by;\\

$\phi_{\eta}(x)=\eta(exp_{\eta}(-\pi i{x\over\eta})-1)$\\

$\psi_{\eta}(x)=\eta(exp_{\eta}(\pi i{x\over \eta})-1)$\\

If $f$ is measurable, we let;\\

$C_{\eta}(t,x)=f(t,{\eta^{2}-1\over \eta})exp_{\eta}(-\pi i{\eta^{2}-1\over \eta}x)-f(t,-\eta)exp_{\eta}(-\pi i(-\eta)x)$\\

$D_{\eta}(t,x)=-{1\over \eta}f(t,-\eta)exp_{\eta}(\pi i{x\over \eta})exp_{\eta}(-\pi i(-\eta)x)$.\\

$C'_{\eta}(t,x)=-{\partial f\over \partial x}(t,-\eta)exp_{\eta}(-\pi i(-\eta)x)$\\

$D'_{\eta}(t,x)=-{1\over \eta}{\partial f\over \partial x}(t,-\eta)exp_{\eta}(\pi i{x\over \eta})exp_{\eta}(-\pi i(-\eta)x)$.\\

$E_{\eta}(t,x)=\phi_{\eta}(x)D_{\eta}(t,x)-C_{\eta}(t,x)$\\

$E'_{\eta}(t,x)=\phi_{\eta}(x)D'_{\eta}(t,x)-C'_{\eta}(t,x)$\\

$F_{\eta}(t,x)=\psi_{\eta}(x)\phi_{\eta}(x)D_{\eta}(t,x)-\psi_{\eta}(x)C_{\eta}(t,x)+\phi_{\eta}(x)D'_{\eta}(t,x)-C'_{\eta}(t,x)$\\

\end{defn}

\begin{rmk}
\label{rmknstransf}
If $f$ is measurable, then so are $\hat{f}$ and $\check{f}$. Again this follows, by transfer, from the finite case, as in Remark \ref{rmkpartials}. By Lemma 0.16 of \cite{dep1}, if $f$ is measurable, then, we have the nonstandard inversion theorems;\\

$\check{\hat{f}}=2f$\\

$\hat{\check{f}}=2f$\\

\end{rmk}

\begin{lemma}
\label{nstransfpartials}
If $f:\overline{\mathcal{T}_{\eta}}\times\overline{\mathcal{R}_{\eta}}\rightarrow{^{*}\mathcal{C}}$ is measurable, then;\\

(i). $\hat{{\partial f\over\partial t}}={\partial{\hat f}\over\partial t}$.\\

(ii). $\hat{{\partial^{2}f\over\partial x^{2}}}=\psi_{\eta}^{2}\hat{f}-F_{\eta}$.\\

\end{lemma}

\begin{proof}

$(i)$. Using Definition \ref{partials} and Definition \ref{nstransf}, we have:\\

$\hat{{\partial f\over\partial t}}(t',y)=\int_{\overline{\mathcal{R}_{\eta}}}{\partial f\over \partial t}(t',x)exp_{\eta}(-\pi i xy)d\lambda_{\eta}(x)$\\

$=\eta(\int_{\overline{\mathcal{R}_{\eta}}}f(t'+{1\over\eta},x)exp_{\eta}(-\pi i xy)d\lambda_{\eta}(x)-\int_{\overline{\mathcal{R}_{\eta}}}f(t',x)exp_{\eta}(-\pi i xy)d\lambda_{\eta}(x))$,\\

$\indent \ \ \ \ \ \ \ \ \ \ \ \ \ \ \ \ \ \ \ \ \ \ \ \ \ \ \ \ \ \ \ \ \ \ \ \ \ \ \ \ \ \ \ \ \ \ \ \ \ \ \ (0\leq t'<{\eta^{2}-1\over\eta})$\\

$=\eta(\hat{f}(t'+{1\over\eta},y)-\hat{f}(t',y))={\partial{\hat f}\over\partial t}(t',y)$,\indent \ \  $(0\leq t'<{\eta^{2}-1\over\eta})$\\

$\hat{{\partial f\over\partial t}}(t',y)={\partial{\hat f}\over\partial t}(t',y)=0$,\indent \ \ \ \ \ \ \ \ \ \ \ \ \ \ \ \ \ \ \ \ \ $({\eta^{2}-1\over\eta}\leq t'<\eta)$\\

$(ii)$. Using the definition of $\psi_{\eta}$ and $F_{\eta}$ in Definition \ref{nstransf}, and the transfer of the result in Lemma 0.21 of \cite{dep1}.\\

\end{proof}

\begin{theorem}
\label{nstransfsol}
Let $f:\overline{\mathcal{T}_{\eta}}\times\overline{\mathcal{R}_{\eta}}\rightarrow{^{*}\mathcal{C}}$ satisfy the conditions of Lemma \ref{nsheat}. Then $\hat{f}$ is determined by;\\

$\hat{f}({i\over\eta},x)=\hat{g}(x)(1+{\psi_{\eta}^{2}(x)\over\eta})^{i}-{1\over\eta}{^{*}\sum}_{0\leq j\leq i-1}F_{\eta}({j\over\eta},x)(1+{\psi_{\eta}^{2}(x)\over\eta})^{i-j-1}$,\\

$\indent \ \ \ \ \ \ \ \ \ \ \ \ \ \ \ \ \ \ \ \ \ \ \ \ \ \ \ \ \ \ \ \ \ \ \ \ \ \ \ \ \ (0\leq i\leq \eta^{2}-1)$, $(*)$\\

In particular, if the boundary condition $g$ satisfies;\\

$g({\eta^{2}-1\over\eta})=0$\\

$g(x)=0$, for $-\eta\leq x< -\eta+\omega$, where $\omega\in{{^{*}\mathcal{N}}\setminus{\mathcal N}}$\\

then, $\hat{f}$ is determined by;\\

$\hat{f}({i\over\eta},x)=\hat{g}(x)(1+{\psi_{\eta}^{2}(x)\over\eta})^{i}$,  $(0\leq i\leq n\eta,n\in{\mathcal N})$\\

\end{theorem}

\begin{proof}
We have that;\\

${\partial f\over\partial t}-{\partial^{2}f\over\partial x^{2}}=0$ on $({\overline{\mathcal{T}_{\eta}}\setminus [{\eta^{2}-1\over\eta},\eta)})\times\overline{\mathcal{R}_{\eta}}$\\

Applying the nonstandard Fourier transform, and using Lemma \ref{nstransfpartials}, we have;\\

${\partial\hat{f}\over\partial t}-(\psi_{\eta}^{2}\hat{f}-F_{\eta})=0$ on $({\overline{\mathcal{T}_{\eta}}\setminus [{\eta^{2}-1\over\eta},\eta)})\times\overline{\mathcal{R}_{\eta}}$\\

Using Definition \ref{partials}, we have;\\

$\eta(\hat{f}({k+1\over\eta},x)-\hat{f}({k\over\eta},x))=\psi_{\eta}^{2}(x)\hat{f}({k\over\eta},x)-F_{\eta}({k\over\eta},x)$\\

$\hat{f}({k+1\over\eta},x)=\hat{f}({k\over\eta},x)(1+{\psi_{\eta}^{2}(x)\over\eta})-{1\over\eta}F_{\eta}({k\over\eta},x)$, $(0\leq k\leq \eta^{2}-2)$. $(**)$\\

Let $A=\{i\in{^{*}\mathcal{N}}:0\leq i\leq \eta^{2}-1,\ for \ which \ $(*)$\ holds\}$. Then $A$ is internal, $A(0)$ holds, as $\hat{f}(0,x)=\hat{g}(x)$, by the boundary condition in Lemma \ref{nsheat}, and if $A(i)$ holds, for $0\leq i\leq \eta^{2}-2$, then, using $(**)$;\\

$\hat{f}({i+1\over\eta},x)=[\hat{g}(x)(1+{\psi_{\eta}^{2}(x)\over\eta})^{i}$\\

$-{1\over\eta}{^{*}\sum}_{0\leq j\leq i-1}F_{\eta}({j\over\eta},x)(1+{\psi_{\eta}^{2}(x)\over\eta})^{i-j-1}](1+{\psi_{\eta}^{2}(x)\over\eta})-{1\over\eta}F_{\eta}({i\over\eta},x)$\\

$=\hat{g}(x)(1+{\psi_{\eta}^{2}(x)\over\eta})^{i+1}-{1\over\eta}{^{*}\sum}_{0\leq j\leq i-1}F_{\eta}({j\over\eta},x)(1+{\psi_{\eta}^{2}(x)\over\eta})^{i-j}-{1\over\eta}F_{\eta}({i\over\eta},x)$\\

$=\hat{g}(x)(1+{\psi_{\eta}^{2}(x)\over\eta})^{i+1}-{1\over\eta}{^{*}\sum}_{0\leq j\leq i}F_{\eta}({j\over\eta},x)(1+{\psi_{\eta}^{2}(x)\over\eta})^{i-j}$\\

so $A(i+1)$ holds. It follows, by hyperfinite induction, see \cite{dep2}, that $A=\{i\in{^{*}\mathcal{N}}:0\leq i\leq \eta^{2}-1\}$, and $\hat{f}$ is determined by the condition $(*)$.\\

Now suppose that the boundary condition $g$ satisfies the requirements in the second part of the Theorem, then, using Lemma \ref{nsheat}, we have;\\

$f({i\over\eta},{\eta^{2}-1\over\eta})=f(0,{\eta^{2}-1\over\eta})=g({\eta^{2}-1\over\eta})=0$, $(0\leq i\leq \eta^{2}-1)$\\

Moreover, again by Lemma \ref{nsheat}, $f({i\over\eta},{-\eta^{2}+1\over\eta})$ and $f({i\over\eta},{-\eta})$ are hyperfinite linear combinations of the values $g({j\over\eta})$, for $-{\eta}^{2}\leq j\leq -{\eta}^{2}+1+2i$. For such $j$, and $0\leq i\leq n\eta$, ${j\over\eta}\leq -\eta+{1+2i\over\eta}\leq -\eta+{1+2n}<-\eta+\omega$, so $g({j\over\eta})=0$ by hypothesis, and, then, $f({i\over\eta},{-\eta^{2}+1\over\eta})=f({i\over\eta},{-\eta})=0$, for
$0\leq i\leq n\eta$. Checking the Definition \ref{nstransf}, it follows that $F_{\eta}({i\over\eta},x)=0$, for $0\leq i\leq n\eta$, $n\in{\mathcal N}$. Then, using the first part of the Theorem, we obtain the final result.

\end{proof}

\begin{defn}{Convolution}\\
\label{convolutiondef}
Suppose that $f,g:\overline{{\mathcal T}_{\eta}}\times\overline{{\mathcal R}_{\eta}}\rightarrow {^{*}\mathcal{C}}$ are measurable. Then we define the nonstandard convolution by;\\

$(f*g)(t,x)=\int_{\overline{{\mathcal R}_{\eta}}}f(t,{[\eta x]\over\eta}-y)g(t,y)d\nu_{\eta}(y)$\\

\end{defn}

\begin{theorem}{Nonstandard Convolution Theorem}\\
\label{convolution}
Let hypotheses be as in \ref{convolutiondef}, then;\\

$\hat{f*g}=\hat{f}\hat{g}$\indent \ \ \  $\check{f*g}=\check{f}\check{g}$\\
\end{theorem}

\begin{proof}
This is a straightforward computation. We have, for $x\in{\overline{{\mathcal R}_{\eta}}}$, using Definition 0.15 of \cite{dep1}, that;\\

$\hat{f}\hat{g}(t,x)={1\over \eta^{2}}[{^{*}\sum}_{j=-\eta^{2}}^{\eta^{2}-1}f(t,{j\over\eta})exp_{\eta}(-\pi i({j\over\eta})x)][{^{*}\sum}_{k=-\eta^{2}}^{\eta^{2}-1}g(t,{k\over\eta})exp_{\eta}(-\pi i({k\over\eta})x)]$\\

$={1\over \eta^{2}}{^{*}\sum}_{j,k=-\eta^{2}}^{\eta^{2}-1}f(t,{j\over\eta})g(t,{k\over\eta})exp_{\eta}(-\pi i({j+k\over \eta})x)$\\

$={1\over \eta^{2}}{^{*}\sum}_{l,k=-\eta^{2}}^{\eta^{2}-1}f(t,{l-k\over\eta})g(t,{k\over\eta})exp_{\eta}(-\pi i({l\over \eta})x)$     $(l=j+k)$\\

$={1\over \eta}{^{*}\sum}_{l=-\eta^{2}}^{\eta^{2}-1}(\int_{\overline{{\mathcal R}_{\eta}}}f(t,{l\over\eta}-w)g(w)d\nu_{\eta}(w))exp_{\eta}(-\pi i({l\over \eta})x)$\\

$={1\over \eta}{^{*}\sum}_{l=-\eta^{2}}^{\eta^{2}-1}(f*g)(t,{l\over \eta})exp_{\eta}(-\pi i({l\over \eta})x)$\\

$=\hat{f*g}(t,x)$\\

A similar calculation shows that $\check{f*g}=\check{f}\check{g}$\\

\end{proof}

\begin{defn}
\label{truncation}
For $\omega'\in{^{*}\mathcal{N}}\setminus{\mathcal{N}}$, with $\omega'<\eta$, we let $F_{\omega'}:\overline{{\mathcal T}_{\eta}}\times\overline{{\mathcal R}_{\eta}}\rightarrow {^{*}\mathcal{R}}$ be the measurable function defined by;\\

$F_{\omega'}(t,{j\over\eta})={1\over 2}$, if $-\omega'\eta\leq j\leq\omega'\eta$\\

$F_{\omega'}(t,{j\over\eta})=0$, $otherwise$\\

and let $F_{\eta}={1\over 2}Id_{\overline{{\mathcal T}_{\eta}}\times\overline{{\mathcal R}_{\eta}}}$\\

\end{defn}

\begin{lemma}
\label{sol}
Let $f$ satisfy the hypotheses of Theorem \ref{nstransfsol}, with the extra requirement on the boundary condition $g$, then, for finite $t$;\\

$\check{F_{\omega'}}*f=(hF_{\omega'})\check{}*g$\\

where $h$ is given by;\\

$h(t,x)=(1+{1\over\eta}\psi_{\eta}(x)^{2})^{[\eta t]}$\\

\end{lemma}

\begin{proof}
By Theorem \ref{nstransfsol}, for finite $t$, $\hat{f}=h\hat{g}$, and so, $\hat{f}F_{\omega'}=hF_{\omega'}\hat{g}$. Let $a=(hF_{\omega'})\check{}$ and $b=\check{F_{\omega'}}$, then, by Theorem \ref{convolution} and Remark \ref{rmknstransf}, $\hat{a*g}=\hat{a}\hat{g}=2hF_{\omega'}\hat{g}=2F_{\omega'}\hat{f}$, and, similarly, $\hat{b*f}=\hat{b}\hat{f}=2F_{\omega'}\hat{f}$. Therefore, $\hat{a*g}=\hat{b*f}$, and, again using Remark \ref{rmknstransf}, we obtain $b*f=a*g$, as required.
\end{proof}

\begin{defn}
\label{nsheatker}
We call $\Psi_{\omega'}(t,x,y)=(hF_{\omega'})\check{}(t,x-y)$ a nonstandard heat kernel on ${\overline{{\mathcal T}_{\eta}}}\times{\overline{{\mathcal R}_{\eta}}}^{2}$.

\end{defn}

\begin{lemma}
\label{specnsheatker}
Let $\Psi(t,x,y)$ be as in Definition \ref{nsheatker}. Then, for finite $(t,x,y)\in\overline{\mathcal{T}_{\eta}}\times\overline{{\mathcal R}_{\eta}}^{2}$, with ${^{\circ}t}\neq 0$, and $\omega'\leq log({\eta})^{1\over 2}$;\\

${^{\circ}\Psi}(t,x,y)={1\over\sqrt{4\pi {^{\circ} t}}}exp({-({^{\circ}x}-{^{\circ}y})^{2}\over 4{^{\circ}t}})$\\

\end{lemma}

\begin{proof}
We first claim that, for finite $x\in{\overline{\mathcal{R}_{\eta}}}$, ${^{\circ}\gamma}_{\eta}(x)=exp(-\pi^{2}{^{\circ} x}^{2})$, where $\gamma_{\eta}(x)=(1+{1\over\eta}\psi_{\eta}(x)^{2})^{\eta}$, $(*)$. For $y\in\mathcal{R}$, let $(s_{n})_{n\in\mathcal{N}}$ be the standard sequence, defined by;\\

$s_{n}(y)={exp(\pi i{y\over n})-1\over {1\over n}}=y({exp(\pi i{y\over n})-1\over {y\over n}})$, $(y\neq 0)$\\

Then, for $y\neq 0$;\\

$lim_{n\rightarrow\infty}(s_{n}(y))=lim_{h\rightarrow 0}y({exp(\pi ih)-1\over h})=y{d\over ds}|_{s=0}exp(\pi is)=\pi iy$\\

and the sequence converges uniformly in $y$, on bounded intervals, (\footnote{\label{roc} We estimate the rate of convergence of the sequence $p_{n}=n(exp(\pi i{1\over n})-1)-i\pi$. We have;\\

 $p_{n}=\sum_{m\geq 2}{(\pi i)^{m}({1\over n})^{m-1}\over m!}={-\pi^{2}\over n}\sum_{m\geq 0}{({\pi i\over n})^{m}\over (m+2)!}$\\

 $|p_{n}|\leq{\pi^{2}\over n}\sum_{m\geq 0}{({\pi\over n})^{m}\over m!}={\pi^{2}\over n}exp({\pi\over n})\leq {\pi^{2}exp(\pi)\over n}$\\

In particular, $|p_{n}|<\epsilon$, and, therefore, $|s_{n}(y)-i\pi y|<\epsilon |y|$, if $n\geq {\pi^{2}exp(\pi)\over \epsilon}$. Hence, $|s_{n}(y)-i\pi y|<\epsilon$, if $n\geq{\pi^{2}exp(\pi)|y|\over\epsilon}$}). Now, it is standard result,(\footnote{\label{roc2} We estimate the rate of convergence of the sequence $q_{n}(w)=r_{n}(w)-exp(w)$, for $w\in{\mathcal C}$. We have, taking a branch of the logarithm with $log(1)=0$, and cutting the complex plane from $-\infty$ to $-1$, for $n>|w|$;\\

$log(r_{n}(w))-w=nlog(1+{w\over n})-w$\\

$|log(r_{n}(w))-w|\leq \sum_{m=1}^{\infty}{|w|^{m+1}\over (m+1)n^{m}}\leq {|w|^{2}\over n}\sum_{m=0}^{\infty}{|w|^{m}\over n^{m}}={|w|^{2}\over n}{1\over 1-{|w|\over n}}={|w|^{2}\over n-|w|}$\\

Moreover, observe that, for $w\in\mathcal{C}$;\\

$|exp(w)-1|\leq\sum_{m=1}^{\infty}{|w|^{m}\over m!}=|w|\sum_{m=0}^{\infty}{|w|^{m}\over (m+1)!}\leq |w|exp(|w|)$\\

Therefore, for $\epsilon>0$, $|exp(w)-1|<\epsilon$, if $|w|<min({\epsilon\over e},1)$, and, for $w',w\in\mathcal{C}$, with $Re(w)\leq 0$,  $|exp(w')-exp(w)|<\epsilon$ if $|exp(w'-w)-1|<\epsilon\leq\epsilon |exp(-w)|$. Hence, $|exp(w')-exp(w)|<\epsilon$, if $|w'-w|<min({\epsilon\over e},1)$, for $Re(w)\leq 0$.\\

So, for $Re(w)\leq 0$, if ${|w|^{2}\over n-|w|}<min({\epsilon\over e},1)$, that is $n>|w|+|w|^{2} max(1,{e\over \epsilon})$, then $|q_{n}(w)|=|r_{n}(w)-exp(w)|<\epsilon$.}), that the sequence of functions $(r_{n}(w))_{n\in\mathcal{N}}$, defined by;\\

$r_{n}(w)=(1+{w\over n})^{n}$\\

converges uniformly to $exp(w)$ on bounded subsets of $\mathcal{C}$. Therefore, if $(t_{n})_{n\in\mathcal{N}}$ is the sequence defined by $t_{n}=r_{n}(s_{n}^{2})$, then;\\

$lim_{n\rightarrow\infty}t_{n}=exp(-\pi^{2}y^{2})$\\

It follows that the sequence of functions $t_{n}(y)$ converges uniformly to $exp(-\pi^{2}y^{2})$ on bounded intervals of $\mathcal {R}$, (\footnote{\label{roc3} We estimate the rate of convergence of the sequence $b_{n}(y)=t_{n}(y)-exp(-\pi^{2}y^{2})$, for $y\in{\mathcal{R}}$, $y\neq 0$. It is a straightforward calculation, to show that, if $|s_{n}(y)-i\pi y|<min(2|y|,{\epsilon\over 3|y|})$, then $|s_{n}^{2}(y)-(-\pi^{2}y^{2})|<\epsilon$. Combining this with the result of footnote \ref{roc}, we obtain that if $n>max({\pi^{2}exp(\pi)\over 2},{3\pi^{2}exp(\pi)y^{2}\over\epsilon})$, $(*)$, then $|s_{n}^{2}(y)-(-\pi^{2}y^{2})|<\epsilon$. Using footnote \ref{roc2}, we also have that if $\epsilon<min({\delta\over e},1)$, then $|exp(s_{n}^{2}(y))-exp(-\pi^{2}y^{2})|<\delta$, $(**)$. Now, assuming $(*)$ is satisfied, we have $|s_{n}(y)|<(\epsilon+\pi^{2}y^{2})^{1\over 2}$. Then, using footnote \ref{roc2}, if $\epsilon<min({\delta\over e},1,{\pi^{2}y^{2}\over 2})$, $(\dag)$, $n>max((\epsilon+\pi^{2}y^{2})^{1\over 2}+(\epsilon+\pi^{2}y^{2})max(1,{e\over\delta}),max({\pi^{2}exp(\pi)\over 2},{3\pi^{2}exp(\pi)y^{2}\over\epsilon}))$, $(\dag\dag)$, then $|exp(s_{n}^{2}(y))-exp(-\pi^{2}y^{2})|<\delta$, $|r_{n}(s_{n}^{2}(y))-exp(s_{n}^{2}(y))|<\delta$, so $|b_{n}(y)|=|t_{n}(y)-exp(-\pi^{2}y^{2})=|r_{n}(s_{n}^{2}(y))-exp(-\pi^{2}y^{2})|<2\delta$, $(***)$. Now, if $\delta<min(e,\pi^{2}y^{2}e)$, we can satisfy $(\dag)$ by taking $\epsilon={\delta\over 2e}$. Substituting into $(\dag\dag)$, we obtain, if $n>max((1+\pi^{2}y^{2})^{1\over 2}+(1+\pi^{2}y^{2}){e\over\delta},{\pi^{2}exp(\pi)\over 2},{6\pi^{2}exp(\pi)y^{2}e\over\delta})$, then $(***)$ holds. Taking $\delta={1\over 2|y|^{r}}$, for $r\in\mathcal{N}$, there exist constants $C_{2},C_{3}>0$, such that, for all $y\in\mathcal{R}$, if $n>max(C_{2},C_{3}|y|^{r+2})$, then $|b_{n}(y)|<{1\over |y|^{r}}$.}) In particular, given $N,\epsilon>0$ standard the statement;\\

$\forall y\leq N\exists M\forall n\geq M(|t_{n}(y)-exp(-\pi^{2}y^{2})|<\epsilon)$\\

is true in $\mathcal{R}$, therefore, by transfer, is true in ${^{*}\mathcal{R}}$. As $\epsilon$ and $N$ were arbitrary, it follows that, for all finite $x\in{\overline{\mathcal R}_{\eta}}$;\\

$\gamma_{\eta}(x)\simeq t_{\eta}(x)\simeq {^{*}exp}(-\pi^{2}x^{2})\simeq exp(-\pi^{2}{^{\circ} x}^{2})$\\

using continuity of $exp$, and Theorem 2.25 of \cite{dep2} or \cite{rob}. Therefore, $(*)$ holds. Now, by continuity of the function $q(w)=w^{s}$, for $s\in\mathcal{R}$, and the fact that ${\eta{^{\circ}t}-[\eta t]\over\eta}\simeq 0$, for finite $t\in{\overline{\mathcal{R}_{\eta}}}$, it follows, again using \cite{rob} or Theorem 2.25 of \cite{dep2}, that, for finite  $(t,x)\in{\overline{\mathcal{T}_{\eta}}}\times{\overline{\mathcal{R}_{\eta}}}$, ${^{\circ}h}(t,x)=exp(-\pi^{2}{^{\circ}t}{^{\circ}x}^{2})$, $(**)$.\\

Now if $\omega''\in{{^{*}\mathcal{R}}\setminus{\mathcal R}}$, with $|\omega''|\leq{\eta}^{1\over 4}$, then, in particular, $\eta>max(C_{2},C_{3}|\omega''|^{3})$, see footnote \ref{roc3}. Hence, we have, by transfer, that;\\

$|t_{\eta}(\omega'')-{^{*}exp(-\pi^{2}\omega''^{2})}|<{1\over |\omega''|}\simeq 0$\\

for infinite $\omega''$. As $lim_{x\rightarrow\infty}exp(-\pi^{2}x^{2})=0$, it is a standard result, see \cite{rob}, that ${^{*}exp(-\pi^{2}\omega''^{2})}\simeq 0$, hence $|t_{\eta}(\omega'')|\simeq 0$, and ${^{\circ}t_{\eta}(\omega'')}=0$. Now, by a similar argument to the above, for finite $t$, with ${^{\circ}t}\neq 0$, we have $h(t,\omega'')\simeq 0$. Combining these results, we have that;\\

 ${^{\circ}hF_{\omega'}}|_{st^{-1}({\mathcal T}_{>0})\times{\overline{\mathcal{R}}_{\eta}}}=st^{*}(exp(-\pi^{2}tx^{2})_{\infty})$ $(\sharp)$\\

for $\omega'\in{^{*}\mathcal{N}\setminus\mathcal{N}}$, with $\omega'\leq{\eta}^{1\over 4}$. Here, we adopt the notation in Definition 0.5 of \cite{dep1}, letting $exp(-\pi^{2}tx^{2})_{\infty}$ denote the extension of $exp(-\pi^{2}tx^{2})$ on $\mathcal{T}_{>0}\times\mathcal{R}$ to $\mathcal{T}_{>0}^{+\infty}\times\mathcal{R}^{+-\infty}$, by setting $exp(-\pi^{2}tx^{2})_{\infty}=0$, at infinite values.\\

Now, for finite $(t,x)\in{\overline{\mathcal{T}_{\eta}}}\times{\overline{\mathcal R}_{\eta}}$, we have, by $(**)$, that;\\

$|h(t,x)|\leq 2{^{*}exp}(-\pi^{2}tx^{2})$ $(\dag)$\\

For $(t,x)\in{\overline{\mathcal{T}_{\eta}}}\times{\overline{\mathcal R}_{\eta}}$, with $t>1$ finite, and $x$ infinite, with $x\leq{\eta}^{1\over 5}$, we have, using footnote \ref{roc3}, that;\\

$|t_{\eta}(x)|^{t}<|t_{\eta}(x)|\leq |t_{\eta}(x)|\leq {^{*}exp(-\pi^{2}x^{2})}+{1\over x^{2}}\leq {C\over x^{2}}$ $(\dag\dag)$\\

where $C\in\mathcal{R}_{>0}$.\\

For $(t,x)\in{\overline{\mathcal{T}_{\eta}}}\times{\overline{\mathcal R}_{\eta}}$, with ${1\over r}<{^{\circ}t}\leq 1$, ${r\in\mathcal{N}}$ and $x$ infinite, with $x\leq{\eta}^{1\over 2r+3}$, we have, by footnote \ref{roc3}, that;\\

$|t_{\eta}(x)|\leq{^{*}exp(-\pi^{2}x^{2})}+{1\over x^{2r}}\leq {C'\over x^{2r}}$\\

$|t_{\eta}(x)|^{t}<|t_{\eta}(x)|^{1\over r}\leq {C''\over x^{2}}$ $(\dag\dag\dag)$\\

where $C',C''\in\mathcal{R}_{>0}$. Combining the estimates, $(\dag),(\dag\dag),(\dag\dag\dag)$, and, using the fact that $h(x,t)$ is the measurable counterpart of $t_{\eta}(x)^{t}$, we have, for $\omega'\leq (log(\eta))^{1\over 2}$, and $t$ finite, $0<{^{\circ}t}$, that;\\

$|(hF_{\omega'})_{t}|\leq f_{t,\eta}$\\

Here, $f_{t,\eta}:{\overline{\mathcal R}_{\eta}}\rightarrow{\overline{\mathcal R}_{\eta}}$ is the measurable counterpart of the ${^{*}}$-continuous function $f_{t}:{^{*}{\mathcal R}}\rightarrow{^{*}{\mathcal R}}$ given by;\\

$f_{t}(x)=C_{t}$ if $|x|\leq 1$\\

$f_{t}(x)={C_{t}\over x^{2}}$ if $|x|>1$\\

and $C_{t}\in\mathcal{R}_{>2}$, depends on $t$. Now, using the proof of Theorem 0.17 in \cite{dep1}, it follows that $f_{t,\eta}$ is $S$-integrable. Then, using \cite{and}, Corollary 5, it follows that $(hF_{\omega'})_{t}(w)$ and $(hF_{\omega'})_{t}(w)exp_{\eta}(\pi iwz)$ are $S$-integrable, $d\lambda_{\eta}(w)$, for finite $z\in{\overline{\mathcal R}_{\eta}}$. Moreover, using \cite{dep2}, Theorem 3.24, and $(\sharp)$, we have, for finite $(t,z)\in \overline{\mathcal{T}_{\eta}}\times\overline{{\mathcal R}_{\eta}}$, with ${^{\circ}t}\neq 0$, that;\\

${^{\circ}(hF_{\omega'})\check{}(t,z)}={^{\circ}\int_{\overline{\mathcal R}_{\eta}}hF_{\omega'}(t,w)exp_{\eta}(\pi i wz)d\lambda_{\eta}(w)}$\\

$={1\over 2}\int_{w finite}{^{\circ}h}exp_{\eta}(\pi iw({^{\circ}z}))dL(\lambda_{\eta})(w)$\\

$={1\over 2}\int_{\mathcal{R}}exp(-\pi^{2}{^{\circ}t}{^{\circ}w}^{2})exp(\pi i w({^{\circ}z}))d\mu(w)$ $(\sharp\sharp)$\\

$={1\over\sqrt{4\pi^{\circ}t}}exp({-({^{\circ}z})^{2}\over 4{^{\circ}t}})$, (\footnote{\label{transform} Taking standard parts, the fact that;\\

$\int_{\mathcal{R}}e^{i\pi wz-\pi^{2}tw^{2}}dw={1\over\sqrt{\pi t}}e^{-z^{2}\over 4t}$\\

is a standard result, which we include for want of a convenient reference. We have $i\pi wz-\pi^{2}tw^{2}=-\pi^{2}t(w-{iz\over 2\pi t})^{2}-{z^{2}\over 4t}$. Hence;\\

$\int_{\mathcal{R}}e^{i\pi wz-\pi^{2}tw^{2}}dw$\\

$=e^{-z^{2}\over 4t}\int_{\mathcal{R}}e^{-\pi^{2}t(w-{iz\over 2\pi t})^{2}}dw$\\

$=e^{-z^{2}\over 4t}\int_{Im(w')={-z\over 2\pi t}}e^{-\pi^{2}t w'^{2}}dw'$ $(w'=w-{iw\over 2\pi t})$\\

$={e^{-z^{2}\over 4t}\over \pi\sqrt{t}}\int_{Im(w'')={-\pi\sqrt{t}z\over 2\pi t}}e^{-w''^{2}}dw''$   $(w''=\pi\sqrt{t}w')$\\

$={1\over\sqrt{\pi t}}e^{-z^{2}\over 4t}$\\}). Now substituting $x-y$ for $z$, we obtain the result.

\end{proof}

\begin{defn}
\label{classical}
Let $g:\mathcal{R}\rightarrow\mathcal{C}$ be a continuous function, satisfying the growth condition;\\

$|g(x)|\leq Aexp(B|x|^{\rho})$, $(x\in\mathcal{R})$\\

for some constants $A,B$ and $\rho<2$. Then the function $H:\mathcal{T}\times\mathcal{R}\rightarrow\mathcal{C}$, defined by;\\

$H(0,x)=g(x)$\\

$H(t,x)={1\over\sqrt{4\pi t}}\int_{\mathcal{R}}exp({-(x-y)^{2}\over 4t})g(y)d\mu(y)$ $(t>0)$\\

which is continuous, and satisfies the standard heat equation;\\

${\partial H\over\partial t}-{\partial^{2}H\over\partial{x}^{2}}=0$\\

on $\mathcal{T}_{>0}\times{\mathcal{R}}$, (\footnote{\label{evans} A good proof of this fact can be found in \cite{SS}, (Theorem 2.1), if $g\in S(\mathcal{R})$. For the more general case, see \cite{evans}}), is known as the classical solution to the heat equation with boundary condition $g$.
\end{defn}

\begin{theorem}
\label{specsol}
Let $g$ be as in Definition \ref{classical}, let $g_{\eta}$ denote its measurable extension to $\overline{\mathcal{R}_{\eta}}$, and, let $g_{\eta,\omega}$ be the truncation of $g_{\eta}$, given by;\\

$g_{\eta,\omega}=g_{\eta}\chi_{[-\omega,\omega)}$\\

for a nonstandard step function $\chi_{[-\omega,\omega)}$, with $\omega\in{{^{*}\mathcal{N}}\setminus\mathcal{N}}$, $\eta-\omega$ infinite and $\omega<\omega'^{1\over 2}$. Then, with $f$ determined by Lemma \ref{nsheat}, for $g_{\eta,\omega}$ as the boundary condition, we have;\\

${^{\circ}(\check{F_{\omega'}}*f)}|_{st^{-1}(\mathcal{T}_{>0}\times\mathcal{R})}=st^{*}(H_{\infty})$\\

 if $\omega'<log(\eta)^{1\over 2}$, and $H_{\infty}$ is obtained from the classical solution $H$ of the heat equation, with boundary condition $g$, given in Definition \ref{classical}.

\end{theorem}

\begin{proof}

Using the following footnote \ref{roc4}, we obtain, by transfer and the measurability observation at the end of Lemma \ref{specnsheatker}, that, for any given $\delta',t\in{^{*}\mathcal{R}}_{>0},x\in{^{*}\mathcal{R}}$;\\

$|h(t,x)-exp_{\eta}(-\pi^{2}tx^{2})|<\delta'$, $(*)$\\

if $\eta>C_{4}(x,\delta',t)$. In particular, observing that the function $C_{4}:{^{*}\mathcal{R}}\times{^{*}\mathcal{R}}_{>0}^{2}$ is increasing in $x$ and $t$, we have, for a given infinite $\omega'\in{^{*}\mathcal{N}}$, that $(*)$ holds for all $|x|\leq \omega'$, and finite $t$, if $\eta>C_{4}(\omega',\delta',\omega')$,(\footnote{\label{roc4} Taking a principal branch of the logarithm, we have, for $w,w'\in\mathcal{C}$, with $w\neq 0$ and $|w'-w|<{|w|\over 2}$, that the function $\theta(t)=log(w+t(w'-w))$ is continuously differentiable on the interval $[0,1]$, with;\\

$\theta'(t)={w'-w\over w+t(w'-w)}$\\

Applying the mean value theorem to the real and imaginary parts of $f$, we obtain;\\

$|log(w)-log(w')|=|\theta(1)-\theta(0)|\leq 2{|w-w'|\over min_{t\in [0,1]}|w+t(w'-w)|}\leq 4{|w-w'|\over |w|}$ $(*)$\\

Using footnote \ref{roc2}, we have, for $w,w'\in\mathcal{C}$, with $|w-w'|<1$ and $Re(w)\leq 0$, that;\\

$|exp(w)-exp(w')|=|exp(w)||exp(w'-w)-1|\leq |w'-w|exp(|w'-w|)|exp(w)|\leq e|w-w'|$ $(**)$\\

Now, for $t\in\mathcal{R}$, $t>0$, we can satisfy the condition $|tlog(w')-tlog(w)|<1$, using $(*)$ and assuming that $|w'-w|<{|w|\over 2}$, by taking $|w'-w|<{|w|\over 4t}$. Then, assuming, that $|w|\leq 1$, so that $Re(tlog(w))\leq 0$, and $w\neq 0$, we have, combining $(*),(**)$, that;\\

$|w'^{t}-w^{t}|=|exp(tlog(w'))-exp(tlog(w))|<4et{|w'-w|\over |w|}$\\

if $|w'-w|<min({|w|\over 2},{|w|\over 4t})$, $(\dag)$. We now estimate the rate of convergence of the sequence $v_{n}(y)=t_{n}(y)^{t}-exp(-\pi^{2}ty^{2})$, for $t\in\mathcal{R}_{>0}$. Let $C(\delta,y)$ be the constant obtained in footnote \ref{roc3}, so that there $(***)$ holds. Then, it is easy to see, using $(\dag)$ and the fact that $0<exp(-\pi^{2}y^{2})\leq 1$, that, if, $n>max(C({exp(-\pi^{2}y^{2})\over 4},y),C({exp(-\pi^{2}y^{2})\over 8t},y),C(\delta'{exp(-\pi^{2}y^{2})\over 8et},y))$, then;\\

$|v_{n}(y)|=|t_{n}(y)^{t}-exp(-\pi^{2}ty^{2})|<\delta'$ $(\dag\dag)$\\

In particular, substituting into the expression for $C(\delta,y)$, we can find constants $C_{2},C_{3}\in\mathcal{R}$, such that $(\dag\dag)$ holds, for;\\

$n>max(C_{2},{C_{3}y^{2}exp(\pi^{2}y^{2})t\over\delta'})=C_{4}(y,\delta',t)$\\
}).\\

In particular, we obtain that;\\

$|(hF_{\omega'})\check{}(t,z)-{\check\theta}(t,z)|\leq\int_{\overline{\mathcal{R}}_{\eta}}|hF_{\omega'}(t,w)-\theta(t,w)|d\lambda(w)$\\

$\leq 2\centerdot{1\over 2}\delta'\omega'=\delta'\omega'$ $(**)$\\

for all $z\in{\overline{\mathcal{R}}_{\eta}}$ and finite $t\in{\overline{\mathcal{T}}_{\eta}}$, $t\neq 0$, where $\theta(t,w)=exp_{\eta}(-\pi^{2}tw^{2})F_{\omega'}(t,w)$. Now, substituting $x-y$ for $z$ in $(**)$, and multiplying through by $g_{\eta,\omega}(y)$, we have, from $(**)$, that;\\

$|(hF_{\omega'})\check{}(t,x-y)g_{\eta,\omega}(y)|\leq |{\check\theta}(t,x-y)g_{\eta,\omega}(y)|+\delta'\omega'|g_{\eta,\omega}(y)|$ $(***)$\\

for all $x,y\in{\overline{\mathcal{R}}_{\eta}}$ and $t$ as above. Now using the growth condition in Definition \ref{classical}, we have that $|\delta'\omega'g_{\eta,\omega}|\leq {\chi_{[-\omega,\omega)}\over \omega^{2}}$, if $\delta'\leq {{^{*}exp}(-B|\omega|^{\rho})\over A\omega^{2}\omega'}$. Using \cite{dep2}(Theorem 3.24), the fact that $\int_{\overline{\mathcal{R}}_{\eta}}{\chi_{[-\omega,\omega)}\over \omega^{2}}d\lambda={2\over\omega}\simeq 0$, and \cite{and}, we have $\delta'\omega'|g_{\eta,\omega}|$ is $S$-integrable, and $\int_{\overline{\mathcal{R}_{\eta}}}\delta'\omega'|g_{\eta,\omega}|d\lambda\simeq 0$. In particular, using Definition \ref{convolutiondef} and Lemma \ref{sol}, this implies that;\\

$(\check{F_{\omega'}}*f)(t,x)=(hF_{\omega'})\check{}*g_{\eta,\omega}(t,x)\simeq\check{\theta}*g_{\eta,\omega}(t,x)$ $(****)$\\

for all finite $t\in\overline{\mathcal{R}}_{\eta,>0}$ and $x\in\mathcal{R}_{\eta}$. Let $\tau(t,w)=exp_{\eta}(-\pi^{2}tw^{2})$, then, using the following footnote \ref{roc5}, we obtain by transfer;\\

$|\check{\tau}(t,z)-{1\over\sqrt{\pi t}}exp_{\eta}({-z^{2}\over 4t})|\leq {K(t)\over\eta}+G(t){|z|\over\eta}+H{z^{2}\over\eta}$\\

for $t\in\overline{\mathcal{R}}_{\eta,>0}$ and $z\in\overline{\mathcal{R}}_{\eta}$,(\footnote{\label{roc5} We require the following estimate, see \cite{dep1} for relevant terminology. Let $f:\mathcal{R}\rightarrow\mathcal{R}$ be differentiable on $\mathcal{R}$ and increasing (decreasing) $(*)$ on the interval $[{i\over n},{j\over n}]$, where $i,j\in\mathcal{Z}$, $-n^{2}\leq i<j\leq n^{2}$, and $n\in\mathcal{N}$, then;\\

$|\int_{[{i\over n},{j\over n}]}f_{n}d\lambda_{n}-\int_{[{i\over n},{j\over n}]}f d\mu|$\\

$\leq{1\over n}\sum_{k=0}^{j-i-1}|f({i+k+1\over n})-f({i+k\over n})|$ $(**)$\\

$={1\over n}\sum_{k=0}^{j-i-1}|\int_{i+k\over n}^{i+k+1\over n}f' d\mu|$ $(***)$\\

$\leq{1\over n}\sum_{k=0}^{j-i-1}\int_{i+k\over n}^{i+k+1\over n}|f'| d\mu$\\

$={1\over n}\int_{i\over n}^{j\over n}|f'|d\mu$ $(****)$\\

where, in $(**)$, we have used the assumption $(*)$ and the definition of the relevant integrals, and, in $(***)$, we have used the Fundamental Theorem of Calculus. Now let $Y(x)=exp(-\pi^{2}tx^{2})cos(\pi xz)$ and let $Y_{n}(x)$ be its $\lambda_{n}(x)$ measurable counterpart on $\mathcal{R}_{n}$, where $t\in\mathcal{R}_{>0}$ and $z={j\over n}$, $0<j\leq n^{2}-1$. Observe that the zeros of $Y$ on $[0,n]$ are located at the points $p_{k}={(2k-1)n\over 2j}$, for $k\in\mathcal{N}\cap [0,j]$, and, the local maxima (minima) of $Y$ on $[0,n]$, are located at points $q_{k}$, where $p_{k}<q_{k}<p_{k+1}$, for $0\leq k\leq j-1$, and $p_{j}<q_{j}<n$, for $n\geq D$, some $D\in\mathcal{R}$. Let $p_{k}'$ denote the points ${[np_{k}]\over n}$, $p_{k}''$ the points $p_{k}'+{1\over n}$, and, similarly, define $q_{k}',q_{k}''$, then, it is easy to see (check this) that we can choose a constant $D(t)$, such that $0<p_{k}'<p_{k}''<q_{k}'<q_{k}''<p_{k+1}'<n$, for $0\leq k\leq j-1$, and $p_{j}'<q_{j}'<q_{j}''<n$, for $n\geq max(D(t),\sqrt{2j})$. Now, using $(****)$, and the fact that $Y$ is monotone on the intervals $[p_{k}'',q_{k}']$, $[q_{k}'',p_{k+1}']$, for $0\leq k\leq j-1$, and on $[0,p_{0}'],[p_{j}'',q_{j}'],[q_{j}'',n]$, we obtain;\\

$|\int_{[p_{k}'',q_{k}')}Y_{n} d\lambda_{n}-\int_{[p_{k}'',q_{k}')}Y d\mu|\leq{1\over n}\int_{[p_{k}'',q_{k}')}|Y'|d\mu$ $(\dag)$\\

and, similarly, for the other intervals. Choose a constant $A(t)\in\mathcal{R}$, such that $|Y(x)|\leq{1\over x^{2}}$, $(\sharp)$, for $|x|>A(t)$. Let $k_{max}$ be the largest $k$ such that $p_{k}''\leq A(t)$, then ${(2k_{max}-1)n\over 2j}+{1\over n}\leq A(t)$ and $k_{max}\leq {jC(t)\over n}+1$. Let $U=\bigcup_{0\leq k\leq k_{max}}[p_{k}',p_{k}'')\cup [q_{k}',q_{k}'')$, then, using the bound $|Y|\leq 1$;\\

$|\int_{U}Y_{n}d\lambda_{n}|\leq{1\over n}2k_{max}\leq{2jC(t)\over n^{2}}+{2\over n}$ $(\dag\dag)$\\

and, similarly, for $|\int_{U}Y d\mu|$. Let $V=\bigcup_{k_{max}<k\leq j}[p_{k}',p_{k}'')\cup [q_{k}',q_{k}'')$, then, using the bound $(\sharp)$;\\

$|\int_{V}Y_{n}d\lambda_{n}|\leq{2\over n}\sum_{k_{max}<k\leq j}{1\over({(2k-1)n\over 2j})^{2}}\leq 16{j^{2}\over n^{3}}$ $(\dag\dag\dag)$\\

and, similarly, for $|\int_{V}Y d\mu|$. Let $W=(\bigcup_{0\leq k\leq j-1}[p_{k}'',q_{k}')\cup[q_{k}'',p_{k+1}'))\cup[0,p_{0}']\cup[p_{j}'',q_{j}']\cup[q_{j}'',n]$. Then, using $(\dag)$, we have;\\

$|\int_{W}Y_{n} d\lambda_{n}-\int_{W}Y d\mu|\leq{1\over n}\int_{W}|Y'|d\mu\leq {1\over n}\int_{[0,n)}|Y'|d\mu$ $(\dag\dag\dag\dag)$\\

Using $(\dag\dag),(\dag\dag\dag),(\dag\dag\dag\dag)$, and the fact that $U,V,W$ is a partition of $[0,n)$, we obtain;\\

$|\int_{[0,n)}Y_{n} d\lambda_{n}-\int_{[0,n)}Y d\mu|\leq {1\over n}\int_{[0,n)}|Y'|d\mu+{4jC(t)\over n^{2}}+{4\over n}+32{j^{2}\over n^{3}}$ $(\sharp\sharp)$\\

Then, as $Y$ is even, $|Y|\leq 1$, $|Y'|\leq exp(-\pi^{2}tx^{2})(2\pi^{2}t|x|+{\pi j\over n})$, we obtain, using $(\sharp\sharp)$;\\

$|\int_{\mathcal{R}_{n}}Y_{n} d\lambda_{n}-\int_{[-n,n]}Y d\mu|$\\

$\leq 2|\int_{[0,n)}Y_{n} d\lambda_{n}-\int_{[0,n)}Y d\mu|+{1\over n}(|Y(-n)|+|Y(0)|)$\\

$\leq {1\over n}\int_{\mathcal{R}}|Y'|d\mu+{8jC(t)\over n^{2}}+{8\over n}+64{j^{2}\over n^{3}}+{2\over n}$\\

$\leq{D(t)\pi j\over n^{2}}+{2E(t)\pi^{2}t\over n}+{8jC(t)\over n^{2}}+{10\over n}+64{j^{2}\over n^{3}}={F(t)\over n}+G(t){j\over n^{2}}+H{j^{2}\over n^{3}}$ $(\sharp\sharp\sharp)$\\

where $F(t),G(t),H\in\mathcal{R}$. Choosing a constant $I(t)\in\mathcal{R}$ such that $exp(-\pi^{2}tx^{2})\leq{I(t)\over 2x^{2}}$, for $|x|>1$, we obtain;\\

$|\int_{\mathcal{R}_{n}}Y_{n} d\lambda_{n}-\int_{\mathcal{R}}Y d\mu|$\\

$\leq {F(t)\over n}+G(t){j\over n^{2}}+H{j^{2}\over n^{3}}+{I(t)\over n}={J(t)\over n}+G(t){j\over n^{2}}+H{j^{2}\over n^{3}}$ $(\sharp\sharp\sharp\sharp)$\\

Let $Z(x)=exp(-\pi^{2}tx^{2})sin(\pi xz)$, with hypotheses and $Z_{n}(x)$ as above. Then, as $Z$ is odd, $|Z|\leq 1$, we have $\int_{\mathcal{R}}Z d\mu=0$, $\int_{\mathcal{R}_{n}}Z_{n} d\lambda_{n}={Z(-n)\over n}$, and;\\

$|\int_{\mathcal{R}_{n}}Z_{n} d\lambda_{n}-\int_{\mathcal{R}}Z d\mu|\leq {1\over n}$ $(\sharp\sharp\sharp\sharp\sharp)$\\

Let $X(x)=exp(-\pi^{2}tx^{2})exp(i\pi xz)$, with hypotheses and $X_{n}(x)$ as above. Then, using the estimates $(\sharp\sharp\sharp\sharp)$, $(\sharp\sharp\sharp\sharp\sharp)$ and footnote \ref{transform}, we obtain;\\

$|\int_{\mathcal{R}_{n}}X_{n} d\lambda_{n}-{1\over\sqrt{\pi t}}e^{-z^{2}\over 4t}|\leq {K(t)\over n}+G(t){j\over n^{2}}+H{j^{2}\over n^{3}}$ $(\sharp\sharp\sharp\sharp\sharp\sharp)$\\

where $K(t)=J(t)+1$ and $z={j\over n}$.}). In particular, if $\omega''\in{^{*}\mathcal{N}}$ is infinite, $\delta''\in{^{*}\mathcal{R}}_{>0}$, then we obtain;\\

$|{1\over 2}\check{\tau}(t,z)-{1\over\sqrt{4\pi t}}exp_{\eta}({-z^{2}\over 4t})|<\delta''$ $(\dag)$\\

for all finite $t\in{\overline{\mathcal{R}}_{\eta}}_{>0}$, and $z\in{\overline{\mathcal{R}}_{\eta}}$, with $|z|\leq\omega''$, if $\eta>{3\omega''^{3}\over\delta''}$, $(\dag\dag)$. Using Definition \ref{nstransf}, and transfer of the following footnote \ref{roc6}, we have;\\

$|{1\over 2}\check{\tau}(t,z)-\check{\theta}(t,z)|\leq {1\over 2}\int_{|x|\geq\omega'}exp_{\eta}(-\pi^{2}tx^{2})d\lambda_{\eta}(x)$\\

$\leq{1\over 2\pi\sqrt{t}}{^{*}exp}(-\pi^{2}t(\omega'-{1\over\eta})^{2})$\\

$\leq\omega'{^{*}exp}(-\pi^{2}t(\omega'-1)^{2})$ $(\dag\dag\dag)$\\

for all $z\in\overline{\mathcal{R}_{\eta}}$, finite $t\in\overline{\mathcal{R}_{\eta}}_{>0}$, $\omega'\in{^{*}\mathcal{N}}$ infinite,(\footnote{\label{roc6} We make the following estimate, with $t\in\mathcal{R}_{>0}$;\\

$\int_{|x|\geq{j\over n},x\in\mathcal{R}_{n}}exp_{n}(-\pi^{2}tx^{2})d\lambda_{n}(x)$\\

$\leq {2\over n}\sum_{k=j}^{n^{2}}exp(-\pi^{2}t({k\over n})^{2})$\\

$\leq 2\int_{{j-1\over n}}^{n}exp(-\pi^{2}tx^{2}d\mu(x)$\\

$\leq 2\int_{{j-1\over n}}^{\infty}exp(-\pi^{2}tx^{2})d\mu(x)$\\

$=2\int_{\pi^{2}t({j-1\over n})^{2}}^{\infty}{exp(-u)\over 2\pi\sqrt{tu}}d\mu(u)$, $(u=\pi^{2}tx^{2})$\\

$\leq {1\over \pi\sqrt{t}}\int_{\pi^{2}t({j-1\over n})^{2}}^{\infty}exp(-u)d\mu(u)$, $({j-1\over n}\geq {1\over\pi\sqrt{t}})$\\

$\leq {1\over\pi\sqrt{t}}exp(-\pi^{2}t({j-1\over n})^{2})$\\}).\\

Combining $(\dag)$ and $(\dag\dag\dag)$, gives;\\

$|\check{\theta}(t,z)-\Gamma(t,z)|\leq \delta''+\omega'{^{*}exp}(-\pi^{2}t(\omega'-1)^{2})$ $(\dag\dag\dag\dag)$\\

for all finite $t\in{\overline{\mathcal{R}}_{\eta}}_{>0}$, and $|z|\leq\omega''$, $z\in{\overline{\mathcal{R}}_{\eta}}$, if the condition $(\dag\dag)$ holds, where $\Gamma(t,z)={1\over\sqrt{4\pi t}}exp_{\eta}({-z^{2}\over 4t})$. We have, using Definition \ref{convolutiondef} and $(\dag\dag\dag\dag)$;\\

$|(\check{\theta}*g_{\eta,\omega})(t,x)-(\Gamma*g_{\eta,\omega})(t,x)|$\\

$=|\int_{{\overline{\mathcal{R}}_{\eta}}}(\check{\theta}-\Gamma)(x-y)g_{\eta,\omega}(y)d\lambda_{\eta}(y)|$\\

$\leq\int_{{\overline{\mathcal{R}}_{\eta}}}(\delta''+\omega'{^{*}exp}(-\pi^{2}t(\omega'-1)^{2}))|g_{\eta,\omega}(y)|d\lambda_{\eta}(y)$ $(\sharp)$\\

for finite $t\in{\overline{\mathcal{R}}_{\eta}}_{>0}$, finite $x\in{\overline{\mathcal{R}}_{\eta}}$, if $\omega''=2\omega$, that is, from $(\dag\dag)$, $\eta>{24\omega^{3}\over\delta''}$, $(\sharp\sharp)$. Following the same argument as above, we have $\delta''g_{\eta,\omega}$ is $S$-integrable and $|\delta''g_{\eta,\omega}|\leq{\chi_{[-\omega,\omega)}\over\omega^{2}}$, if $\delta''\leq {{^{*}}exp(-B\omega^{\rho})\over\omega^{2}}$, so we require, from $(\sharp\sharp)$, that $\eta>24\omega^{5}{^{*}exp}(B\omega^{\rho})$, $(\sharp\sharp\sharp)$. Similarly, we have $\omega'{^{*}exp}(-\pi^{2}t(\omega'-1)^{2}))g_{\eta,\omega}$ is $S$-integrable and $|\omega'{^{*}exp}(-\pi^{2}t(\omega'-1)^{2}))g_{\eta,\omega}|\leq{\chi_{[-\omega,\omega)}\over\omega^{2}}$, if ${^{*}exp}(-\pi^{2}t(\omega'-1)^{2}+1)\leq {{^{*}}exp(-B\omega^{\rho})\over\omega^{2}}$. By a simple calculation, this can be achieved if $\omega'\geq Cmax(log(\omega)^{1\over 2},\omega^{\rho+1\over 2})$, $(\sharp\sharp\sharp\sharp)$. If both the conditions $(\sharp\sharp\sharp)$ and $(\sharp\sharp\sharp\sharp)$ are satisfied, we then have;\\

$(\check{\theta}*g_{\eta,\omega})(t,x)\simeq (\Gamma*g_{\eta,\omega})(t,x)$ $(\sharp\sharp\sharp\sharp\sharp)$\\

for finite $t\in{\overline{\mathcal{R}}_{\eta}}_{>0}$, finite $x\in{\overline{\mathcal{R}}_{\eta}}$. Finally, using Definition \ref{convolutiondef}, we have;\\

$\Gamma*g_{\eta,\omega}(t,x)=\int_{{\overline{\mathcal{R}}_{\eta}}}\Gamma(t,x-y)g_{\eta,\omega}(y)d\lambda_{\eta}(y)$ $(!)$\\

By the growth condition on $g$, for $x\in\mathcal{R}$, $t\in\mathcal{R}_{>0}$, if $\Psi(t,x-y)$ denotes the standard heat kernel, the function $\Psi(t,x-y)g(y):\mathcal{R}\rightarrow\mathcal{C}$ is continuous and satisfies the tail estimate $|\Psi(t,x-y)g(y)|\leq {1\over y^{2}}$ for sufficiently large $|y|\geq A(t)$, $A(t)\in\mathcal{R}$. Using the proof of Theorem 0.17 in \cite{dep1} and Theorem 3.24 of \cite{dep2}, we obtain that $\Gamma(t,x-y)g_{\eta,\omega}(y)$ is $S$-integrable and ${^{\circ}(\Gamma*g_{\eta,\omega})}(t,x)=H(t,x)$, $(!!)$. For finite $x\in\overline{\mathcal{R}_{\eta}}$, finite $t\in\overline{\mathcal{R}_{\eta,>0}}$, and ${^{\circ}t}>0$, we have that $\Gamma(t,x-y)g_{\eta,\omega}(y)$ is $S$-integrable, $(!!!)$. In order to see this, choose $0<t_{1}<t<t_{2}$, with $t_{1},t_{2}\in\mathcal{R}$, and $x_{1}<x<x_{2}$, with $x_{1},x_{2}\in\mathcal{R}$. We then have;\\

$\Gamma(t,x-y)|g_{\eta,\omega}(y)|\leq \sqrt{t_{1}\over t_{2}}\Gamma(t_{2},x_{1}-y)|g_{\eta,\omega}(y)|$, for $y\leq x_{1}$\\

$\Gamma(t,x-y)|g_{\eta,\omega}(y)|\leq \sqrt{t_{2}\over t_{1}}\Gamma(t_{2},x_{2}-y)|g_{\eta,\omega}(y)|$, for $y\geq x_{2}$\\

$\Gamma(t,x-y)|g_{\eta,\omega}(y)|\leq C(t)$, for $x_{1}\leq y\leq x_{2}$ $(!!!!)$\\

where $C(t)\in\mathcal{R}$, and in $(!!!!)$, we have used the fact that $g$ is continuous. Now applying the result of $(!!)$ and using \cite{and} (Corollary 5), we obtain $(!!!)$. Then, again using Theorem 3.24 of \cite{dep2}, we have that ${^{\circ}(\Gamma*g_{\eta,\omega})}(t,x)=H({^{\circ}t},{^{\circ}x})$, $(!!!!!)$. Combining $(****)$,$(\sharp\sharp\sharp\sharp\sharp)$ and $(!!!!!)$, we obtain that;\\

${^{\circ}(\check{F_{\omega'}}*f)}(t,x)=H({^{\circ}t},{^{\circ}x})$ $(A)$\\

for finite $x\in\overline{\mathcal{R}_{\eta}}$, finite $t\in\overline{\mathcal{R}_{\eta}}_{>0}$, under the conditions;\\

$\eta>max(C_{4}(\omega',{{^{*}exp}(-B|\omega|^{\rho})\over \omega^{2}\omega'},\omega'),25\omega^{5}{^{*}exp}(B|\omega|^{\rho}))$ $(B)$\\

$\omega'>C max({^{*}log}(\omega)^{1\over 2},\omega^{\rho+1\over 2})$ $(C)$\\

By a simple calculation, we can satisfy $(B),(C)$ if;\\

$\eta>\omega^{5}\omega'^{4}{^{*}}exp(B\omega^{\rho})$, $\omega'>\omega^{2}$ $(D)$\\

and $(D)$ if;\\

$\eta>\omega'^{6}exp(B\omega')$, $\omega'>\omega^{2}$ $(E)$\\

Therefore, it is sufficient to have;\\

$\omega'<log(\eta)^{1\over 2}$, $\omega<\omega'^{1\over 2}$ $(F)$.\\

Using $(A)$ and condition $(F)$, we obtain the result.

\end{proof}

\begin{theorem}
\label{derivs}

Let $g$ be as in Definition \ref{classical}, let $g_{\eta}$ denote its measurable extension to $\overline{\mathcal{R}_{\eta}}$, and, let $g_{\eta,\omega}$ be the truncation of $g_{\eta}$, given by;\\

$g_{\eta,\omega}=g_{\eta}\chi_{[-\omega,\omega)}$\\

for a nonstandard step function $\chi_{[-\omega,\omega)}$, with $\omega\in{{^{*}\mathcal{N}}\setminus\mathcal{N}}$, $\eta-\omega$ infinite and $\omega<\omega'^{1\over 2}$. Then, with $\hat{f}$ determined by Theorem \ref{nstransfsol}, with $g_{\eta,\omega}$ as the boundary condition, we have;\\

${^{\circ}(\ {}^{\check{}}({F_{\omega'}\hat{f}}))}|_{st^{-1}(\mathcal{T}_{>0}\times\mathcal{R})}=st^{*}(H_{\infty})$\\

if $\omega'<log(\eta)^{1\over 2}$, and $H_{\infty}$ is obtained from the classical solution $H$ of the heat equation, with boundary condition $g$, given in Definition \ref{classical}.

 \end{theorem}

\begin{proof}

With notation as above, we have that;\\

$\check{F_{\omega'}}*f={1\over 2}(\check{F_{\omega'}}*(\check{\hat{f}}))$ (by Theorem \ref{rmknstransf})\\

$\check{F_{\omega'}}*(\check{\hat{f}})=^{\check{}}({2F_{\omega'}\hat{f}})$\\

as, by Theorem \ref{rmknstransf} and Remarks \ref{convolutiondef}, for $a,b: \overline{\mathcal{T}}_{\eta}\times\overline{\mathcal{R}}_{\eta}\rightarrow{^{*}\mathcal{C}}$, ${^{\hat{}}}(\check{a}*\check{b})=\hat{\check{a}}\hat{\check{b}}=2a.2b=4ab$, $(*)$, and, $2(\check{a}*\check{b})={^{\check{}}}\ {^{\hat{}}}(\check{a}*\check{b})={^{\check{}}}(4ab)$, by $(*)$ and Theorem \ref{rmknstransf}. Therefore;\\

$\check{F_{\omega'}}*f={}^{\check{}}(F_{\omega'}\hat{f})$\\

and the result follows by Theorem \ref{specsol}. 

\end{proof}

\begin{rmk}
\label{timescales}
Theorem \ref{derivs} gives a solution to the heat equation, obtained by the following steps;\\

(i). Truncating the transfer of the boundary data.\\

(ii). Taking the nonstandard Fourier transform of this data and solving the resulting ODE in Theorem \ref{nstransfsol}.\\

(iii). Truncating the solution again.\\

(iv). Taking the inverse nonstandard Fourier transform.\\

(v). Specialising.\\

By straightforward results on limits in nonstandard analysis, see Theorem 2.22 of \cite{dep2}, it follows that the above algorithm converges for $\{m,n,n'\}$, with $n<(n')^{1\over 2}$, $n'<log(m)^{1\over 2}$, (replacing $\{\eta,\omega,\omega'\}$ respectively), as $m\rightarrow\infty$ (noting that, for $\eta$ infinite, $\eta-log(\eta)^{1\over 4}$ is infinite). It seems likely that the algorithm is faster than current methods involving a recursion over both the space and time steps. However, this still has to be decided computationally.

\end{rmk}

\end{document}